\documentclass{amsart}
\usepackage{amssymb,amsmath,amsthm}
\usepackage{graphicx}
\title[Representations of a Pair of Complex Structures]{Representation Theory of the Algebra Generated \\By a Pair of Complex Structures}
\author{Steven Gindi}
\date{\today}

\newtheorem{lemma}{Lemma}[section]
\newtheorem{thm}{Theorem}[section]
\newtheorem{cor}{Corollary}[section]
\newtheorem{rmk}{Remark}[section]
\theoremstyle{definition} \newtheorem{example}{Example}[section]
\theoremstyle{definition}  \newtheorem{defi}{Definition}[section]

\numberwithin{equation}{section}

\begin{document}
\begin{large}

\begin{abstract}
The objective of this paper is to determine the finite dimensional, indecomposable representations of the algebra that is generated by two complex structures over the real numbers. Since the generators satisfy relations that are similar to those of the infinite dihedral group, we give the algebra the name $iD_{\infty}$.
\end{abstract}

\maketitle
 \section{Introduction: Complex structures and $iD_{\infty}$} 
 
The goal of this paper is to classify the finite dimensional, indecomposable representations of the algebra over 
 the real numbers that is generated by two complex structures, $J_{1}$ and $J_{2}$ (so that $J_{1}^{2}=J_{2}^{2} =-1$). 
 A simple and familiar example of this algebra is when the complex structures anticommute, leaving us with the quaternion algebra. But in general, given two complex structures, there may not be such a simple relation between them. The aim then is to find the representations of the algebra generated by any 
 $J_{1}$ and $J_{2}$.

   The first step in obtaining the representations is to rewrite the generators of the algebra as follows. 
   Let 
   \begin{equation*}
    a=J_{1}J_{2}
   \end{equation*}
   and   
    \begin{equation*}
      b= J_{2}.
   \end{equation*}
   Clearly, both $\left\{ J_{1} , J_{2}\right\}$ and $\left\{ a,b \right\}$ generate the same algebra; but it was found that the latter was easier to use when deriving the representations. Upon realizing that $\left\{ a,b \right\}$  satisfy  relations that are similar to those satisfied by the generators of the infinite dihedral group, $D_{\infty}$, we are led to the following definition. 
\begin{defi}
Let $iD_{\infty}$ be  the algebra over $\mathbb{R}$ generated by two elements, $a$ and $b$, that satisfy the following relations:
  \begin{equation*}
 bab^{-1}=a^{-1} 
\end{equation*}
  \begin{equation*}
 b^{2}=-1. 
\end{equation*}
\end{defi}
The goal is to find the indecomposable representations of $iD_{\infty}.$ Here is the main theorem:
\begin{thm} The finite dimensional, indecomposable, real representations, V, of $iD_{\mathbb{\infty}}$ are of the  form:
\begin{equation*}
V= \mathbb{R}[t]/(p^{n}) \oplus b\mathbb{R}[t]/(p^{n}), 
\end{equation*}  
where $n \in \mathbb{N}$, $a$ acts by multiplication by t and 
\begin{enumerate}
  \item   $ p= t-r$ $(r \in \mathbb{R}-\left\{0 \right\})$
  \item  $ p= (t-c)(t-\overline{c})$ $(c \in \mathbb {C}$ but not in $\mathbb{R}).$
\end{enumerate}
 \label{thm1.1}
 \end{thm}
Knowing the form of the indecomposable representations, we can now choose an appropriate basis for $V $ and obtain the following corollary.
 \begin{cor}
  (a)  Let $p=t-r$ ($r \neq 0)$ and $ V= \mathbb{R}[t]/(p^{n}) \oplus b\mathbb{R}[t]/(p^{n})$. We may choose a basis for V and represent the elements $a$ and $b$ in terms of two $ n \times n$ matrices, $A$ and $1$ (the identity):   
\begin {equation*}
a=\left(\begin{array}{cc}A & 0 \\0 & A^{-1} \end{array}\right)  and  \ b=\left(\begin{array}{cc}0 & -1_{n \times n} \\ 1_{n \times n} & 0\end{array}\right)
\end{equation*}
where 
\begin{equation*}
A = \left(\begin{array}{cccccc}r & 1 & 0 & \cdots & 0 & 0 \\0 & r & 1 & \cdots & 0 & 0 \\  \vdots & \vdots & \vdots & \ddots & \vdots& \vdots \\0 & 0 & 0 & \vdots & r & 1 \\0 & 0 & 0 & \vdots & 0 & r\end{array}\right).
\end{equation*}
(b)
Let $p=(t-c)(t-\overline{c})$ where $ c=e+if $  $( f \neq 0)$ and let $ V=\mathbb{R}[t]/(p^{n}) \oplus b\mathbb{R}[t]/(p^{n}) $. We may choose a basis for V and represent the elements $a$ and $b$ in terms of two $ 2n \times 2n$ matrices, $A$ and $1$ (the identity):
\begin {equation*}
a=\left(\begin{array}{cc}A & 0 \\0 & A^{-1} \end{array}\right)  and  \ b=\left(\begin{array}{cc}0 & -1_{2n \times 2n} \\ 1_{2n \times 2n} & 0\end{array}\right)
\end{equation*}
where 
\begin{equation*}
A= \left(\begin{array}{cccccc}D & 1_{2 \times 2} & 0 & \cdots & 0 & 0 \\0 & D & 1_{2 \times 2} & \cdots & 0 & 0 \\  \vdots & \vdots & \vdots & \ddots & \vdots& \vdots \\0 & 0 & 0 & \vdots & D & 1_{2 \times 2} \\0 & 0 & 0 & \vdots & 0 & D\end{array}\right) and \ D=\left(\begin{array}{cc}e & -f \\f & e\end{array}\right). 
\end{equation*}
\label{cor1.1}
\end{cor}
\begin{proof}[ Proof of Corollary 1.1]\
 
(a) The basis that leads to the action of $a$ via $A$ in $\mathbb{R}[t]/(p^{n})$ is simply the standard Jordan basis. Let $\left\{ e_{i} \right\}_{(1 \leq i \leq n)}$ denote this basis. Then the basis  $\left\{ be_{i} \right\}_{(1 \leq i \leq n)}$ was chosen for the $\mathbb{R}[a]$-submodule, $b\mathbb{R}[t]/(p^{n})$. Note that $a^{-1}$ acts in this basis via $A$; thus $a$ acts via $A^{-1}.$ \

(b) The proof is the same as that for (a) except that the standard real Jordan basis was used instead of the regular Jordan basis (see \cite{alg}). \end{proof}
 One may compare the above result to the indecomposable representations of the infinite dihedral group (see \cite{dih}).\\
 
Next, using Theorem \ref{thm1.1}, we may classify the irreducible representations of $iD_{\infty}$.
\begin{cor} The irreducible representations, V, of $iD_{\infty}$ are those as stated in Theorem \ref{thm1.1} but with n=1.
That is, 
\begin{equation*}
V= \mathbb{R}[t]/(p) \oplus b\mathbb{R}[t]/(p).  
\end{equation*} 
\label{cor1.2}
\end{cor}
\begin{proof} First we show that $V$ is irreducible. We need to consider two cases. 

\begin{enumerate}
  \item   When $p=t-r$, $V= \mathbb{R}[t]/(p) \oplus b\mathbb{R}[t]/(p)$ is a two dimensional vector space over $\mathbb{R}$. One cannot find a one dimensional invariant subspace in $V$ because $b$ does not have any real eigenvalues.
  \item  When $p=(t-c)(t-\overline{c})$, $V= \mathbb{R}[t]/(p) \oplus b\mathbb{R}[t]/(p)$ is a four dimensional vector space over $\mathbb{R}$. Since it is known that every $\mathbb{R}[b]$-module is even  dimensional, it suffices to show that there cannot exist a two dimensional $iD_{\infty}$-submodule, $W$, in $V$. To see this, note that by Theorem \ref{thm1.1}, such a $W$ must be isomorphic to $\mathbb{R}[t]/(p') \oplus b\mathbb{R}[t]/(p')$, where $p'=t-r$. But this contradicts the fact that $p(t)p(t^{-1})=0$ on $W.$ (As one may check, $p(t)$ and $p(t^{-1})$ are zero on $\mathbb{R}[t]/(p)$ and $b\mathbb{R}[t]/(p),$ respectively.)
    \end{enumerate}

    Let us now show that, for $n>1$, $V= \mathbb{R}[t]/(p^{n}) \oplus b\mathbb{R}[t]/(p^{n})$ is reducible. Indeed $(\overline{p^{n-1}})\oplus b(\overline{p^{n-1}})$, where $(\overline{p^{n-1}})$ is the ideal in $\mathbb{R}[t]/(p^{n}),$ is an $iD_{\infty}$-submodule.

\end{proof}
To get a feeling for some of the above representations, let us consider a simple example where $iD_{\infty}$ acts by the   quaternion algebra.
\begin{example} Let $V= \mathbb{R}[t]/(p) \oplus b\mathbb{R}[t]/(p)$, $p=(t-c)(t-\overline{c})$ where $c\overline{c}=1$ and $c=e+if$. If we define $J=f^{-1}(t-e)$, one may then easily check that $iD_{\infty}$ is generated by two anticommuting elements: $\left\{b,J\right\}$ and thus acts on $V$ by the quaternion algebra, $\mathbb{H}$. Also since $\left\{1,J, b\cdot1,bJ\right\}$ is a basis for $V$ over $\mathbb{R}$, it is clear that $V$ itself is isomorphic to $\mathbb{H}$; this also follows from the fact that $\mathbb{H}$ is simple. The quaternions will appear again in Section 3.3 to play an important role in deriving the representations of $iD_{\infty}$.   
\end{example}
Let us now turn to the task of proving Theorem \ref{thm1.1}. First, here is a useful definition.
\begin{defi} 
 Let $ p \in \mathbb{R}[t]$ be a monic, irreducible polynomial over $ \mathbb{R}$. 
 \begin{enumerate}
  \item   If  $ p= t-r$ ($r \in \mathbb{R}-\left\{0 \right\})$, then define $\tilde{p} =Ê ( t-r^{-1} )$.
  \item  ÊIf $ p= (t-c)(t-\overline{c})$ ($c \in \mathbb {C}$ but not in $\mathbb{R}$) then define \\ $\tilde{p}= (t-c^{-1})(t-\overline{c^{-1}}).$

\end{enumerate}
\end{defi}

To begin proving the theorem, let us consider a representation, $V$, of $iD_{\infty}$. The first step in understanding the action of $iD_{\infty}$ on $V$ is to decompose $V$ into its primary decomposition over $a$:
\begin{equation}
V= \bigoplus_{p}{V_{p}}.
\end{equation}
Here we are considering $V$ as an $\mathbb{R}[t]$-module, where $ t \cdot v= av$, for $v \in V$. The element $p$ is irreducible in $\mathbb{R}[t]$ and $V_{p}=\left\{v \in V |  p^{n}v=0 \right\},$ for some $n$ $\in \mathbb{N}.$

Knowing that each $V_{p}$ is an $\mathbb{R}[a]$-submodule, let us consider the action of $b$. From the relation $ bab^{-1}=a^{-1}$, it easily follows that $ b(V_{p})=V_{\tilde{p}}$. Since we are interested in the indecomposable representations of $iD_{\infty}$, it is clear that there are two possibilities for $V$ that we need to consider:
\begin{align*}
   & \text{ (1)}\  V=V_{p} \oplus V_{\tilde{p}}, \text{ where} \  p \neq \tilde{p} \\
   & \text{ (2)}\  V=V_{p}, \text{ where} \  p = \tilde{p}. 
\end{align*}

We explore these cases in the following sections.
\section{$V=V_{p} \oplus V_{\tilde{p}}$ $(p \neq \tilde{p})$}
     Using the structure theorem for modules over PIDs and the fact that $V_{\tilde{p}}=bV_{p}$, the module $V$ can be expressed as a direct sum of $\mathbb{R}[a]$-submodules.
 \begin{lemma} We may decompose $V$ in the following manner:
 \begin{equation*}
V= \bigoplus_{ 1 \leq j \leq m}\mathbb{R}[t]/(p^{n})w_{j} \oplus b\mathbb{R}[t]/(p^{n})w_{j}, 
\end{equation*}
for some $\left\{w_{j}\right\}$ in $V_{p}$.
\end{lemma}
  As one may check, the relevant indecomposable $iD_{\infty}$-modules are \\ $\mathbb{R}[t]/(p^{i}) \oplus b\mathbb{R}[t]/(p^{i}),$ for $i \in \mathbb{N}.$ This proves Theorem \ref{thm1.1} for the case when $p \neq \tilde{p}$.\\


\section{ $V=V_{p} $ $(p = \tilde{p})$}

    When considering the case $p = \tilde{p}$, the first question that arises is how to find the action of $b$ on $ V=V_p$. When $V$ was equal to $V_{p} \oplus bV_{{p}}$, the action of $b$ was clear; however, it is less clear how $b$ acts in the present case. The goal then is to find a method that will allow us to decompose $V=V_{p}$ into a direct sum of indecomposable $iD_{\infty}$-submodules so that the action of $b$ will become manifest.

  Of course, the general classification theorem of modules over PIDs tells us that we can always, in particular, decompose $V_{p}$ into a direct sum of $\mathbb{R}[t]$-submodules. But to make these submodules indecomposable and the action of $b$ apparent, requires that we have more control over how we decompose $V_{p}$. The method used to accomplish these goals will be presented in Section 3.1 and applied to the cases when $p=t-r$ $(r=\pm1)$ and $ p= (t-c)(t-\overline{c})$ $(c\overline{c}=1)$ in Sections 3.2 and 3.3, respectively. 

\subsection{Decomposing $V=V_{p}$ by choosing a basis in $V_{k}/V_{k-1}$  } 

Let us begin with a necessary definition.
\begin{defi}
Choose $n \in \mathbb{N}$ so that $p^{n}$ acts by zero on V and define $V_{k}$ \\$(1 \leq k \leq n)$ as the $\mathbb{R}[t]$-submodule in $V$ such that $p^{k}V_{k} = 0$. 
\end{defi}
Next, let us note that $\mathbb{R}[t]/(p) $ is a field; it is isomorphic to $\mathbb{R}$ or $\mathbb{C}$ when $p=t-r$ or $ p= (t-c)(t-\overline{c})$, respectively. Also, since $p(V_{k}/V_{k-1})=0$, $V_{k}/V_{k-1}$ is a vector space over $\mathbb{R}[t]/(p).$
The goal of this section is to show that once one chooses an appropriate set of bases for the vector spaces $\left\{V_{k}/V_{k-1}\right\}_{(1 \leq k \leq n)}$ over $\mathbb{R}[t]/(p) $, one obtains a decomposition of $V=V_{p}$ into a direct sum of submodules.  
 
\begin{lemma}
If $\left\{w_{1},...w_{s} \right\}$ are elements of $V_{k}$ such that $\left\{\overline{w}_{1},...\overline{w}_{s} \right\}$ are linearly independent in the vector space $V_{k}/V_{k-1}$, over $\mathbb{R}[t]/(p) $, then $\mathbb{R}[t]/(p^{n})w_{1} $ is disjoint from $\mathbb{R}[t]/(p^{n})w_{2}+...+\mathbb{R}[t]/(p^{n})w_{s} $. Consequently, $\mathbb{R}[t]/(p^{n})w_{1}\oplus...\oplus\mathbb{R}[t]/(p^{n})w_{s} $ is an $\mathbb{R}[t]$-submodule in $V$.
\end{lemma}
\begin{proof}
Assume that 
\begin{equation}
f_{1}(t)w_{1}+f_{2}(t)w_{2}+...+f_{s}(t)w_{s}=0,
\label{eq1}
\end{equation}
where $\left\{f_{i}(t)\right\}$ are elements of $\mathbb{R}[t]/(p^{n}) $. Write 
$f_{i}(t)=p^{m_{i}}f'_{i}(t)$, where $f'_{i}(t)$ does not contain any factors of $p$ in its prime decomposition. Without loss of generality, let $m_{i}\geq m_{1}$ for all $i$. Now multiply (\ref{eq1}) by $p^{k-m_{1}-1}$; since $\left\{w_{1},...w_{s} \right\}$ are elements of $V_{k}$, we obtain:
 \begin{equation}
p^{k-1}(f'_{1}(t)w_{1}+...+f'_{q}(t)w_{q})=0.
\label{eq2}
\end{equation}
This implies that $(f'_{1}(t)\overline w_{1}+...+f'_{q}(t)\overline w_{q})=0 $ in the module $V_{k}/V_{k-1}$. Since the $\left\{f'_{i}(t)\right\}$ are non-zero in $\mathbb{R}[t]/(p)$, this contradicts the linear independence of $\left\{\overline{w}_{1},...\overline{w}_{s} \right\}$ in $V_{k}/V_{k-1}$.
\end{proof}
 
\begin{lemma}
For each k $(1 \leq k \leq n)$, let $\left\{w_{k,1},...w_{k,m_{k}} \right\}$ be elements in $V_{k}$ such that 
$\left\{\overline{p^{i-k}w_{i,j}}\right\}_{\substack{k \leq i \leq n \\ 1 \leq j \leq m_{i}}}$
is a basis for $V_{k}/V_{k-1}$ over the field $\mathbb{R}[t]/(p) $. Then 
\begin{equation*}
\bigoplus_{\substack{k \leq i \leq n \\ 1 \leq j \leq m_{i}}}
\mathbb{R}[t]/(p^{n})w_{i,j}
\end{equation*}
 is a submodule in $V$.
\label{lem3} 
\end{lemma}

\begin{proof}
Suppose we have the following relation:  
\begin{equation}
\sum_{\substack{k \leq i \leq n \\ 1 \leq j \leq m_{i}}}
f_{i,j}(t)w_{i,j}=0 ,
\label{eq3}
\end{equation}
where the elements $\left\{f_{i,j}(t)\right\}$ are in $\mathbb{R}[t]/(p^{n})$. If we now mod out (\ref{eq3}) by $V_{n-1}$, we find:
\begin{equation}
f_{n,1}(t)\overline{w}_{n,1} +...+ f_{n,m_{n}}(t)\overline{w}_{n,m_{n}}=0.
\end{equation}
Since $\left\{\overline{w}_{n,1},...\overline{w}_{n,m_{n}} \right\}$ are linearly independent in $V_{n}/V_{n-1}$ over 
$\mathbb{R}[t]/(p)$, it follows that we may write $(f_{n,j}(t)=pf'_{n,j}(t))$ for ${(1 \leq j \leq m_{n})}$.
Similarly, if we then mod out (\ref{eq3}) by $V_{n-2}$, we find that $(f'_{n,j}(t)=pf''_{n,j}(t))$ for ${(1 \leq j \leq m_{n})}$ and $(f_{n-1,i}(t)=pf'_{n-1,i}(t))$ for ${(1 \leq i \leq m_{n-1})}$. Continuing in this manner, it is clear that we may rewrite (\ref{eq3}) in the following form: 
\begin{equation}
\sum_{\substack{k \leq i \leq n \\ 1 \leq j \leq m_{i}}}
\widetilde{f}_{i,j}(t)p^{i-k}w_{i,j}=0,
\end{equation}
where $\left\{\widetilde{f}_{i,j}(t)\right\}$ are in $\mathbb{R}[t]/(p^{n})$. This, however, contradicts the linear independence of 
$\left\{\overline{p^{i-k}w_{i,j}}\right\}_{\substack{k \leq i \leq n \\ 1 \leq j \leq m_{i}}}$ in $V_{k}/V_{k-1}$.
\end{proof}
 
Since it is easy to see that the module, 
\begin{equation*}
\bigoplus_{\substack{1 \leq i \leq n \\ 1 \leq j \leq m_{i}}}
\mathbb{R}[t]/(p^{n})w_{i,j}
\end{equation*}
spans all of V, we are led to the following lemma.
 
 
\begin{lemma}Let $\left\{w_{k,1},...w_{k,m_{k}} \right\}$ be elements in $V_{k}$ as in Lemma \ref{lem3}. Then 
\begin{equation*}
V=\bigoplus_{\substack{1 \leq i \leq n \\ 1 \leq j \leq m_{i}}}
\mathbb{R}[t]/(p^{n})w_{i,j}.
\end{equation*}
\label{lem4}
\end{lemma}
\subsection{ Decomposing $V_{p}$ when $p=t-r$ $(r=\pm1)$ }
Using the above lemma, we can now proceed to find the action of $iD_{\infty}$ on $V=V_{p}$. The goal is to construct a set of simple bases for $\left\{V_{k}/V_{k-1}\right\}_{(1 \leq k \leq n)}$ over $\mathbb{R}[t]/(p),$ that will, by Lemma \ref{lem4}, allow us to express $V$ as a direct sum of indecomposable $iD_{\infty}$-submodules as well as make the action of $b$ apparent. 

An important fact to note is that since $p=\widetilde{p}$, $V_{k}/V_{k-1}$ is itself an $iD_{\infty}$-module. We state this as a lemma:
\begin{lemma} Given that $p=\widetilde{p}$, $V_{k}/V_{k-1}$ is an $iD_{\infty}$-module.
\label{lem3.4}
\end{lemma} 
\begin{proof}
One may check that the $\mathbb{R}[a]$-module, $V_{k}$, is also an $\mathbb{R}[b]$-module for the cases when (1) $p=t-r$, where $r=\pm1$ and (2) $p=(t-c)(t-\overline{c})$, where $c\overline{c}=1$. This in turn implies that $V_{k}/V_{k-1}$ is an $iD_{\infty}$-module.
\end{proof}
 
Let us now proceed to decompose $V$ as a direct sum of indecomposable $iD_{\infty}$-submodules. We first consider the case when $p=t-r $ $(r=\pm1)$; here is the main result:
\begin{lemma}
Let $p=t-r$ $(r=\pm1)$. There exists elements $\left\{w_{1},...w_{m} \right\}$ in V such that
\begin{equation*}
V=\bigoplus_{1 \leq j \leq m} \mathbb{R}[t]/(p^{n})w_{j}\oplus\mathbb{R}[t]/(p^{n})bw_{j}. 
\label{lem5}
\end{equation*}
\end{lemma}
To prove the above lemma, we need the following standard result: 
\begin{lemma} Let W be a real even dimensional vector space and let b be an endomorphism of W such that $b^{2}=-1$. There exists elements $ \left\{w_{j}, bw_{j}\right\}_{1 \leq j \leq m}$ that is a basis for W over $\mathbb{R}$.
\label{lem3.6}
\end{lemma}
\begin{proof} Choose $w_{1}$ to be nonzero in W and consider the real subspace generated by $w_{1}$ and $bw_{1}$: $\left\langle w_{1},bw_{1}\right\rangle_{\mathbb{R}}$. Since this is an $\mathbb{R}[b]$-submodule, by the semisimplicity of $b$, there exists another $\mathbb{R}[b]$-submodule, M, such that $W= \left\langle w_{1},bw_{1}\right\rangle_{\mathbb{R}}\oplus M$. By proceeding similarly in M, it is clear by induction that we may find a basis for W in the following form: $ \left\{w_{j}, bw_{j}\right\}_{1 \leq j \leq m}$.
\end{proof}
 Knowing the above lemma, we may now prove Lemma \ref{lem5}.
\begin{proof} [Proof of Lemma \ref{lem5}] 
As we learned in Lemma \ref{lem4}, to decompose $V$ as a direct sum of cyclic submodules, one should find a suitable set of bases for $\left\{V_{k}/V_{k-1}\right\}_ {(1 \leq k \leq n)}$ over $\mathbb{R}[t]/(p) \cong \mathbb{R}.$ 
Let us begin by considering $V_{n}/V_{n-1}.$ Since $V_{n}/V_{n-1}$ is an $\mathbb{R}[b]$-module, by Lemma \ref{lem3.6}, we may find a set of elements $ \left\{\overline{w}_{n,j}, \overline{bw}_{n,j}\right\}_{{1 \leq j \leq m_{n}}}$ that is a basis for $V_{n}/V_{n-1}$ over $\mathbb{R}$.   
 
Next, consider $V_{n-1}/V_{n-2}$. Note that the elements $\left\{\overline{pw}_{n,j}, \overline{pbw}_{n,j}\right\}_{1 \leq j \leq m_{n}}$ are linearly independent in $V_{n-1}/V_{n-2}$ and since $b$ is semisimple, there exists an $\mathbb{R}[b]$-submodule that is complementary to the submodule generated by these elements. Thus, using the previous lemma, we may find a basis for $V_{n-1}/V_{n-2}$ in the following form:
\begin{equation*}
\left\{\overline{p^{i-(n-1)}w_{i,j}}, \overline{p^{i-(n-1)}bw_{i,j}}\right\}_{\substack{n-1 \leq i \leq n \\ 1 \leq j \leq m_{i}}},
\end{equation*}
where $\left\{\overline{w_{n-1,j}}, \overline{bw_{n-1,j}}\right\}_{ 1 \leq j \leq m_{n-1}}$
are in $V_{n-1}/V_{n-2}$.
 
Proceeding in this manner (always using the semisimplicity of $b$), we find that for each $k$ $(1 \leq k \leq n)$, there exists elements $\left\{w_{k,j},bw_{k,j} \right\}_{1 \leq j \leq m_{k}}$  in $V_{k}$ such that $\left\{\overline{p^{i-k}w_{i,j}},\overline{p^{i-k}bw_{i,j}}\right\}_{\substack{k \leq i \leq n \\ 1 \leq j \leq m_{i}}}$ is a basis for $V_{k}/V_{k-1}$ over $\mathbb{R}$. Thus by Lemma \ref{lem4}, we conclude that
\begin{equation*}
V=\bigoplus_{\substack{1 \leq i \leq n \\ 1 \leq j \leq m_{i}}}
\mathbb{R}[t]/(p^{n})w_{i,j} \oplus \mathbb{R}[t]/(p^{n})bw_{i,j}.
\end{equation*} 
\end{proof}

   As one may check, the relevant indecomposable $iD_{\infty}$-modules are \\ $\mathbb{R}[t]/(p^{i}) \oplus b\mathbb{R}[t]/(p^{i}),$ for $i \in \mathbb{N}.$ This proves Theorem \ref{thm1.1} for the case when $p=t-r$ $(r=\pm1).$  

\subsection{Decomposing $V_{p}$ when $ p= (t-c)(t-\overline{c})$ $(c\overline{c}=1)$ } 
Now let us turn to the case when $p=(t-c)(t-\overline{c})$.  Here is the main result:
\begin{lemma} Let  $p=(t-c)(t-\overline{c})$, where $c\overline{c}=1$. There exists elements $\left\{w_{1},...w_{m} \right\}$ in $V=V_{p}$ such that
\begin{equation*}
V=\bigoplus_{1 \leq j \leq m}\mathbb{R}[t]/(p^{n})w_{j}\oplus\mathbb{R}[t]/(p^{n})bw_{j}.
\end{equation*} 
\label{lem3.1}
\end{lemma}

To prove the above lemma,  we use the methods of Section 3.1 and first find a simple basis for $V_{k}/V_{k-1}$ over $\mathbb{R}[t]/(p) \cong \mathbb{C}$.  Recalling Lemma \ref{lem3.4}, let us begin by exploring how  $iD_{\infty}$ acts on $V_{k}/V_{k-1}$. 

\begin{lemma} The element $a$ acts  on the complex vector space $V_{k}/V_{k-1}$  by some complex number $e^{i\theta}$ while $b$ is a conjugate linear map over $\mathbb{C}$, i.e. $bz=\overline{z}b$ for $z \in \mathbb{C}$.
\end{lemma}
\begin{proof} Since $p=0$ on $V_{k}/V_{k-1}$, $a= c$ or $a=\overline{c}$.
That $b$ is conjugate linear follows simply from the relation $ba=a^{-1}b$.  
\end{proof}
\begin{rmk} Although $b$ is conjugate linear over $\mathbb{C}$, it is of course linear over $\mathbb{R} $ by assumption. 
\end{rmk}

As a first attempt to find a simple basis for $V_{k}/V_{k-1}$ over $\mathbb{C}$, let us try to follow the idea behind the proof of Lemma \ref{lem3.6}. We begin by considering the subspace generated by some elements $\left\{w,bw\right\}$ over $\mathbb{C}$. Since this is an $\mathbb{R}[b]$-submodule (but not a $\mathbb{C}[b]$-submodule, because $b$ is conjugate linear), we may use the semisimplicty of $b$ to find a complementary $\mathbb{R}[b]$-submodule, $W$, such that $V_{k}/V_{k-1}=\left\langle w,bw\right\rangle_{\mathbb{C}}\oplus W $. However in general, $W$ may not be a vector space over $\mathbb{C}$ and thus we cannot proceed by induction to decompose $V_{k}/V_{k-1}$.  Rather to find a useful basis for $V_{k}/V_{k-1}$ over $\mathbb{R}[t]/(p) \cong \mathbb{C}$ we need something stronger than just the semisimplicty of $b$.  

 Indeed it is very interesting to note that $V_{k}/V_{k-1}$ is not only an $\mathbb{R}[b]$-module,  but also a quaternion module. As will be shown, this will be sufficient to find a simple basis for $V_{k}/V_{k-1}$ over $\mathbb{C}$.

\begin{lemma}
Let J be the linear map over  $\mathbb{R}$ defined by $Jv=iv$, for $v \in V_{k}/V_{k-1}$.  The set  $\left\{b,J \right\}$  are generators for the quaternions, $\mathbb{H}$; consequently, $V_{k}/V_{k-1}$ is an $\mathbb{ H}$-module. 
 \end{lemma}
\begin{proof} We need only show that $bJ=-Jb$. But this is immediate from the fact that $b$ is conjugate linear over $\mathbb{C}$.  
\end{proof}

Since $V_{k}/V_{k-1}$ is an $\mathbb{H}$-module, there exists a standard basis for $V_{k}/V_{k-1}$ over $\mathbb{C}$ (Lemma 3.11). For completeness, we provide a simple proof that parallels that of Lemma 3.6. First it is important to note the following.

\begin{lemma} $b$ has no eigenvalues in $V_{k}/V_{k-1}$ (even complex ones!). 
\end{lemma}
\begin{proof}
Suppose we found a $w \in V_{k}/V_{k-1}$ such that $bw=\lambda w$, where $\lambda \in \mathbb{C}$.  We then arrive at a contradiction: $-w=b^{2}w=b\lambda w=\overline{\lambda} bw=\left|\lambda\right|^{2}w$.
\end{proof}

\begin{lemma}  One may choose a basis for $V_{k}/V_{k-1}$ over $\mathbb{R}[t]/(p) \cong \mathbb{C}$ in the following form: $ \left\{w_{j}, bw_{j}\right\}_{1 \leq j \leq m}$.
\label{lem3.5}
\end{lemma}
\begin{proof} Choose $w_{1}$ to be a nonzero vector in $V_{k}/V_{k-1}$ and consider the $\mathbb{H}$-submodule generated by $w_{1}$: $\left\langle w_{1}\right\rangle_{\mathbb{H}}$.  By semisimplicity of $\mathbb{H}$, there exists another $\mathbb{H}$-module, M, such that  $V_{k}/V_{k-1} = \left\langle w_{1}\right\rangle_{\mathbb{H}}\oplus M$.  Thus by induction we may write 
\begin {equation} 
V_{k}/V_{k-1}=\bigoplus_{{1 \leq j \leq m }}  \left\langle w_{j}\right\rangle_{\mathbb{H}}.
\end{equation}
Now, since $b$ does not have any eigenvalues, a basis for $\left\langle w_{j}\right\rangle_{\mathbb{H}}$ over $\mathbb{C}$ is just   $\left\{w_{j}, bw_{j} \right\}$. We thus conclude that $ \left\{w_{j}, bw_{j}\right\}_{1 \leq j \leq m}$ is a basis for $V_{k}/V_{k-1}$ over $\mathbb{C}$.
\end{proof}

Now that we have found a simple basis for $V_{k}/V_{k-1}$, we can proceed to the proof of Lemma \ref{lem3.1} and thereby decompose $V$ into a direct sum of indecomposable $iD_{\infty}$-submodules. 
\begin{proof} [Proof of Lemma \ref{lem3.1}]  We proceed analogously to the case when $p$ was equal to $t-r$. First let us consider the complex vector space $V_{n}/V_{n-1}$. By Lemma \ref{lem3.5}, we may find a set of elements $ \left\{\overline{w}_{n,j}, \overline{bw}_{n,j}\right\}_{1 \leq j \leq m_{n}}$ that is a basis for $V_{n}/V_{n-1}$ over $\mathbb{C}$.  
 
 Next, consider  $V_{n-1}/V_{n-2}$. Note that $ \left\{\overline{pw}_{n,j}, \overline{pbw}_{n,j}\right\}_{1 \leq j \leq m_{n}}$
are linearly independent in $V_{n-1}/V_{n-2}$ and since these elements generate an $\mathbb{H}$-submodule, by semisimplicity, there exists a complementary $\mathbb{H}$-submodule. Thus, by Lemma \ref{lem3.5} we may find a basis for $V_{n-1}/V_{n-2}$ in the following form:
\begin{equation*}
\left\{\overline{p^{i-(n-1)}w_{i,j}}, \overline{p^{i-(n-1)}bw_{i,j}}\right\}_{\substack{n-1 \leq i \leq n \\ 1 \leq j \leq m_{i}}},
\end{equation*}
where $\left\{\overline{w_{n-1,j}}, \overline{bw_{n-1,j}}\right\}_{ 1 \leq j \leq m_{n-1}}$
are in $V_{n-1}/V_{n-2}$.
 
Proceeding in this manner (always using the semisimplicity of $\mathbb{H}$), we find that for each $k$ $(1 \leq k \leq n)$, there exists elements 
$\left\{w_{k,j},bw_{k,j} \right\}_{ 1 \leq j \leq m_{k}}$ in $V_{k}$ such that $\left\{\overline{p^{i-k}w_{i,j}}, \overline{p^{i-k}bw_{i,j}}\right\}_{\substack{k \leq i \leq n \\ 1 \leq j \leq m_{i}}}$ is a basis for $V_{k}/V_{k-1}$ over $\mathbb{C}$. Thus by Lemma \ref{lem4}, we conclude that
\begin{equation*}
V=\bigoplus_{\substack{1 \leq i \leq n \\ 1 \leq j \leq m_{i}}}
\mathbb{R}[t]/(p^{n})w_{i,j} \oplus \mathbb{R}[t]/(p^{n})bw_{i,j}.
\end{equation*} 
\end{proof}

  As one may check, the relevant indecomposable $iD_{\infty}$-modules are \\ $\mathbb{R}[t]/(p^{i}) \oplus b\mathbb{R}[t]/(p^{i}),$ for $i \in \mathbb{N}.$ \\

That completes the proof of Theorem 1.1, classifying the finite dimensional, indecomposable representations of $iD_{\mathbb{\infty}}$.

\section*{Acknowledgements}
I would like to thank Alexander Kirillov for his guidance and helpful discussions and Martin Rocek for discussions as well as proposing the problem of finding the representations of $iD_{\infty}.$   
I would also like to thank Melvin Irizarry for help with \LaTeX.

\textsc{Department of Physics, Stony Brook University, Stony Brook, \\     NY 11794, USA} \\

\textit{E-mail Address:} \texttt{sgindi@ic.sunysb.edu} 

\end{large}


\begin{thebibliography}{9}

\bibitem{alg}
W. Adkins, and S. Weintraub, \textit{Algebra: An Approach via Module Theory,} Springer-Verlag, New York, 1992. 

\bibitem{dih}
 D. Dokovic, \textit{Pairs of Involutions in the General Linear Group,} Journal of Algbera \textbf{100} (1986), 214-223.\\

\end{thebibliography}
\end{document}